\documentclass[reqno,12pt]{amsart}

\usepackage[utf8]{inputenc}
\usepackage{microtype}
\usepackage{amsfonts}
\usepackage{amsmath}
\usepackage{graphicx}
\usepackage{float}
\usepackage{color}
\usepackage{hyperref}
\usepackage{tikz-cd}

\usepackage{amssymb}
\usepackage{enumitem}
\usepackage{mathrsfs}

\usepackage{tikz}
\usetikzlibrary{topaths}
\usetikzlibrary{calc}

\usepackage{a4wide}

\usepackage{empheq}

\newtheorem{theorem}{Theorem}[section]
\newtheorem*{theorem*}{Theorem}

\newtheorem{lemma}[theorem]{Lemma}

\newtheorem{proposition}[theorem]{Proposition}
\newtheorem*{proposition*}{Proposition}
\newtheorem{corollary}[theorem]{Corollary}
\newtheorem*{corollary*}{Corollary}
\newtheorem{conjecture}[theorem]{Conjecture}
\newtheorem{cit}[theorem]{Citation}
\newtheorem*{conjecture*}{Conjecture}

\newtheorem*{main:rationally_full}{Corollary~\ref{cor:rationally_full}}

\theoremstyle{definition}
\newtheorem{definition}[theorem]{Definition}
\newtheorem*{definition*}{Definition}

\newtheorem{observation}[theorem]{Observation}

\newtheorem*{question*}{Question}

\newcommand{\Z}{\mathbb{Z}}
\newcommand{\N}{\mathbb{N}}
\newcommand{\Q}{\mathbb{Q}}
\newcommand{\R}{\mathbb{R}}

\newcommand{\defeq}{\mathbin{\vcentcolon =}}

\DeclareMathOperator{\Hom}{Hom}

\DeclareMathOperator{\Homeo}{Homeo}

\DeclareMathOperator{\F}{F}

\DeclareMathOperator{\Supp}{Supp}

\usepackage{ifthen}
\usepackage{xargs}
\newcommand{\optionalarg}[2]{
\ifthenelse{\equal{#2}{}}{%
#1}{%
#1(#2)}
}

\newcommand{\Thomp}[1]%
   {\optionalarg{\mathscr{T}}{#1}}                 % Thompson functor
   
\newcommand{\Thkern}[1]%
   {\optionalarg{\mathscr{K}}{#1}}                 % Thompson kernel

\newcommand{\Stein}[1]%
   {\optionalarg{\mathscr{X}}{#1}}                 % Stein ``functor''

\newcommand{\Fbr}%
   {F_{\operatorname{br}}}                 % braided F

\newcommand{\Vbr}%
   {V_{\operatorname{br}}}                 % braided V
   
\numberwithin{equation}{section}

\begin{document}

\title{The BNSR-invariants of the Stein group $F_{2,3}$}
\date{\today}
 % Topological methods in group theory


\author[R.~Spahn]{Robert Spahn}
\address{Department of Mathematics and Statistics, University at Albany (SUNY), Albany, NY 12222}
\email{rspahn@albany.edu}

\author[M.~C.~B.~Zaremsky]{Matthew C.~B.~Zaremsky}
\address{Department of Mathematics and Statistics, University at Albany (SUNY), Albany, NY 12222}
\email{mzaremsky@albany.edu}

\begin{abstract}
The Stein group $F_{2,3}$ is the group of orientation-preserving homeomorphisms of the unit interval with slopes of the form $2^p3^q$ ($p,q\in\Z$) and breakpoints in $\Z[\frac{1}{6}]$. This is a natural relative of Thompson's group $F$. In this paper we compute the Bieri--Neumann--Strebel--Renz (BNSR) invariants $\Sigma^m(F_{2,3})$ of the Stein group for all $m\in\N$. A consequence of our computation is that (as with $F$) every finitely presented normal subgroup of $F_{2,3}$ is of type $\F_\infty$. Another, more surprising, consequence is that (unlike $F$) the kernel of any map $F_{2,3}\to\Z$ is of type $\F_\infty$, even though there exist maps $F_{2,3}\to \Z^2$ whose kernels are not even finitely generated. In terms of BNSR-invariants, this means that every discrete character lies in $\Sigma^\infty(F_{2,3})$, but there exist (non-discrete) characters that do not even lie in $\Sigma^1(F_{2,3})$. To the best of our knowledge, $F_{2,3}$ is the first group whose BNSR-invariants are known exhibiting these properties.
\end{abstract}

\maketitle
\thispagestyle{empty}

\section*{Introduction}

Thompson's group $F$ is the group of orientation-preserving piecewise linear homeomorphisms of the unit interval $[0,1]$ with breakpoints in $\Z[\frac{1}{2}]$ and slopes powers of $2$. One can generalize this to a group denoted $F_n$ by considering breakpoints in $\Z[\frac{1}{n}]$ and slopes powers of $n$. Generalizing in another direction, we get the main group of study in this paper:

\begin{definition*}
The \emph{Stein group} $F_{2,3}$ is the subgroup of $\Homeo([0,1])$ consisting of all orientation-preserving piecewise linear homeomorphisms with breakpoints in $\Z[\frac{1}{6}]$ and slopes in the multiplicative group $\langle 2,3\rangle_{\Q^\times}=\{2^p3^q\mid p,q\in\Z\}\le \Q^\times$.
\end{definition*}

The group $F_{2,3}$ was first studied in depth by Stein in \cite{stein92}, along with a family of generalizations allowing more breakpoints and slopes (see Subsection~\ref{sec:questions} for more on this). In certain ways $F_{2,3}$ is more unusual than Thompson's group $F$. For example, combining the construction of $F_{2,3}$ with a result of Bieri--Neumann--Strebel in \cite{bieri87}, it is clear that every kernel of a map from $F_{2,3}$ to $\Z$ is finitely generated (unlike for $F$), but there exist maps from $F_{2,3}$ to $\Z^2$ whose kernels are not finitely generated. Another unexpected fact is that $F_{2,3}$ does not embed as a subgroup into $F$ \cite{lodha20}. Other notable work on the Stein group includes Wladis's work on distortion \cite{wladis11} and metric properties \cite{wladis12}.

An important result Stein proved in \cite{stein92} is that $F_{2,3}$ is of type $\F_\infty$. Recall that a group is of \emph{type $F_n$} if it admits a classifying space with compact $n$-skeleton, and of \emph{type $F_\infty$} if it is of type $\F_n$ for all $n$. Every group is of type $\F_0$, type $\F_1$ is equivalent to finite generation, and type $\F_2$ is equivalent to finite presentability. The fact that Thompson's group $F$ is of type $\F_\infty$, proved by Brown and Geoghegan in \cite{brown84}, made it the first torsion-free group of type $\F_\infty$ that does not admit a compact classifying space, and $F_{2,3}$ enjoys this distinction as well.

For a group $G$ of type $\F_\infty$, one can define the Bieri--Neumann--Strebel--Renz (BNSR) invariants $\Sigma^m(G)$ for $m\in\N$. These are a family of geometric invariants, developed by Bieri--Neumann--Strebel \cite{bieri87} and Bieri--Renz \cite{bieri88}, encoding information about $G$. In particular they reveal the finiteness properties of any subgroup $H\le G$ containing the commutator subgroup $[G,G]$. The BNSR-invariants of $F$ were computed by Bieri--Geoghegan--Kochloukova in \cite{bieri10}, and more generally for the generalized Thompson groups $F_n$ by Kochloukova and the second author in \cite{kochloukova12} and \cite{zaremsky17F_n}.

The main result of this paper is a full computation of $\Sigma^m(F_{2,3})$ for all $m$. The statement of the computation is quite technical, so rather than write it all out here in the introduction we just refer the reader to Theorem~\ref{thrm:main}. A detail-free version of Theorem~\ref{thrm:main} is:

\begin{theorem*}
We have that $\Sigma^1(F_{2,3})$ is a $3$-sphere with two points removed, $\Sigma^2(F_{2,3})$ is $\Sigma^1(F_{2,3})$ with the convex hull of these two points additionally removed, and $\Sigma^m(F_{2,3})=\Sigma^2(F_{2,3})$ for all $m\ge 2$.
\end{theorem*}

One easy-to-state consequence of the computation is the following, which is contained in Corollary~\ref{cor:main}.

\begin{corollary*}
Let $H$ be a subgroup of $F_{2,3}$ containing the commutator subgroup $[F_{2,3},F_{2,3}]$ (which is equivalent to saying $H$ is a non-trivial normal subgroup). If $H$ is finitely presented then it is of type $\F_\infty$.
\end{corollary*}

This is ``expected behavior'' for relatives of $F$. We also get the following unusual result, which makes $F_{2,3}$ very different than $F$:

\begin{main:rationally_full}
The kernel of any character $\chi\colon F_{2,3} \to \Z$ is of type $\F_\infty$. However, there exist characters $\chi\colon F_{2,3} \to \Z^2$ whose kernels are not even finitely generated.
\end{main:rationally_full}

This follows from the precise statement of Theorem~\ref{thrm:main}, which reveals that all discrete character classes lie in $\Sigma^\infty(F_{2,3})$ but there exist (non-discrete) character classes that do not even lie in $\Sigma^1(F_{2,3})$. Another notable consequence is that $F_{2,3}$ provides an example of a type $\F_\infty$ group whose $\Sigma^\infty$ is determined by a polytope, but not an integral polytope; see \cite{kielak20} for a discussion of BNSR-invariants and integral polytopes. To the best of our knowledge $F_{2,3}$ is the first group whose BNSR-invariants are known exhibiting any of this behavior.

This paper is organized as follows. In Section~\ref{sec:background} we establish some background on BNSR-invariants and $F_{2,3}$. In Section~\ref{sec:chars} we set up a certain character basis of $\Hom(F_{2,3},\R)$ and recall the computation of $\Sigma^1(F_{2,3})$ from \cite{bieri87}. The main section is Section~\ref{sec:computation}, where we compute $\Sigma^m(F_{2,3})$ for all $m$. Finally, we discuss some applications in Section~\ref{sec:apps}, along with a discussion of how much more difficult it might be to compute the BNSR-invariants of the other generalizations from \cite{stein92}.

\subsection*{Acknowledgments} We are grateful to Dawid Kielak for encouraging us to undertake this project. The second author is supported by grant \#635763 from the Simons Foundation.

%-------------------------------------------------
\section{Background}\label{sec:background}

In this section we recall some background, first on BNSR-invariants, and second on the action of $F_{2,3}$ on $[0,1]$.

\subsection{BNSR-invariants}\label{sec:bnsr}

The Bieri--Neumann--Strebel--Renz invariants $\Sigma^m(G)$ of a group are a family of geometric invariants, defined whenever $G$ is of type $\F_m$, that encode information about the subgroups of $G$ containing the commutator subgroup $[G,G]$. Each $\Sigma^m(G)$ is a subset of the \emph{character sphere} $\Sigma(G)$ of $G$, that is the projectivization of the euclidean space $\Hom(G,\R)$. Thus an element of $\Sigma(G)$ is an equivalence class $[\chi]$ for $\chi$ a non-trivial \emph{character} of $G$, i.e., an element of $\Hom(G,\R)$, with the equivalence relation given by positive scaling. The invariants are nested:
\[
\Sigma(G)=\Sigma^0(G)\supseteq \Sigma^1(G)\supseteq \Sigma^2(G)\supseteq\cdots \text{.}
\]
If $G$ is of type $\F_\infty$ then $\Sigma^m(G)$ is defined for all $m$, and we can define $\Sigma^\infty(G)\defeq \bigcap_{m\in\N}\Sigma^m(G)$. (As a quick aside: One important fact about the $\Sigma^m(G)$ is that they are open \cite{bieri87,bieri88}, but we remark here that, to the best of our knowledge, it is unknown whether $\Sigma^\infty(G)$ must be open for all $G$.)

The formal definition of $\Sigma^m(G)$ is quite complicated, and in fact we will not need to use it in this paper, but for completeness we recall it here.

\begin{definition}[BNSR-invariants]\label{def:bnsr}
Let $G$ be a group of type $\F_m$. Let $Y$ be an $(m-1)$-connected CW-complex on which $G$ acts freely and cocompactly by cell-permuting homeomorphisms (this exists since $G$ is of type $\F_m$). Let $\chi\colon G\to\R$ be a non-trivial character of $G$, and let $h_\chi\colon Y\to\R$ be a \emph{$\chi$-equivariant} map, that is, satisfying $h_\chi(g.y)=\chi(g)+h_\chi(y)$ for all $g\in G$ and $y\in Y$. For $t\in\R$ let $Y_{\chi\ge t}$ be the full subcomplex of $Y$ spanned by vertices $y$ with $h_\chi(y)\ge t$. Now the $m$th \emph{Bieri--Neumann--Strebel--Renz (BNSR) invariant} $\Sigma^m(G)$ is the subset of $\Sigma(G)$ consisting of all $[\chi]$ such that the filtration $(Y_{\chi\ge t})_{t\in\R}$ is \emph{essentially $(m-1)$-connected}, i.e., for all $t$ there exists $s\le t$ such that the inclusion $Y_{\chi\ge t} \to Y_{\chi\ge s}$ induces the trivial map in homotopy groups $\pi_k$ for all $k\le m-1$.
\end{definition}

The main application of the BNSR-invariants (which is much easier to understand than the definition) is the following:

\begin{cit}\cite[Theorem~1.1]{bieri10}\label{cit:bnsr_fin_props}
Let $G$ be a group of type $\F_n$ and $H$ a subgroup of $G$ containing the commutator subgroup $[G,G]$. Then $H$ is of type $\F_n$ if and only if $[\chi]\in\Sigma^n(G)$ for all $0\ne\chi\in\Hom(G,\R)$ satisfying $\chi(H)=0$.
\end{cit}

In particular if $H$ is the kernel of a map $\chi$ onto a copy of $\Z$, i.e., a \emph{discrete character}, then $H$ is of type $\F_n$ if and only if $[\pm\chi]\in\Sigma^m(G)$. Note that every discrete character is equivalent under positive scaling to one whose image in $\R$ equals $\Z$.

\subsection{Action of the Stein group}\label{sec:stein}

The results in this subsection are analogs for $F_{2,3}$ of well known facts about Thompson's group $F$. Recall that the \emph{support} of a homeomorphism $f\colon [0,1]\to[0,1]$ is $\Supp(f)\defeq\{a\in [0,1]\mid f(a)\ne a\}$.

\begin{lemma}\label{lem:normal_contain_comm}
Every non-trivial normal subgroup of $F_{2,3}$ contains the commutator subgroup $[F_{2,3},F_{2,3}]$.
\end{lemma}

\begin{proof}
The proof is very similar to that of the analogous result for $F$ in \cite[Theorem~4.3]{cannon96}, and our proof here will make use of some well known facts about $F\le F_{2,3}$. First we claim the center of $F_{2,3}$ is trivial. The key is that if two elements commute then they stabilize each other's sets of fixed points in $[0,1]$. Since there exists an element of $F$ with support $(1/2,1)$, any central element $z\in F_{2,3}$ must fix $1/2$. Then since for any dyadic $a\in(0,1)\cap\Z[1/2]$ there exists an element of $F$ taking $a$ to $1/2$, every such $a$ must be fixed by $z$. This implies $z=1$, so the center of $F_{2,3}$ is trivial. Now we need to show that the normal closure $N$ of any non-trivial element $f$ of $F_{2,3}$ contains $[F_{2,3},F_{2,3}]$. Since the center is trivial, we can choose some $g\in F_{2,3}$ such that $[f,g]\ne 1$. Since $f\in N$ implies $[f,g]\in N$, we know $N$ contains a non-trivial element of $[F_{2,3},F_{2,3}]$. Since $[F_{2,3},F_{2,3}]$ is simple by \cite{stein92}, $N$ must contain all of $[F_{2,3},F_{2,3}]$.
\end{proof}

For an interval $[a,b]$ in $\R$, let $F_{2,3}[a,b]$ be the group of orientation-preserving piecewise linear homeomorphisms of $[a,b]$ with breakpoints in $\Z[\frac{1}{6}]$ and slopes in $\langle 2,3\rangle_{\Q^\times}$, so for example $F_{2,3}=F_{2,3}[0,1]$.

\begin{lemma}\label{lem:conjugator}
The group $F_{2,3}[0,1]$ is conjugate in $F_{2,3}[-1,1]$ to $F_{2,3}[1/2,1]$. In particular $F_{2,3}[1/2,1]\cong F_{2,3}$.
\end{lemma}

\begin{proof}
Let $f$ be any element of $F_{2,3}[-1,1]$ that maps $[0,1]$ bijectively to $[1/2,1]$, for example
\[
f(x)\defeq
\begin{cases}
\frac{3}{2}x+\frac{1}{2} & \text{ if } -1\le x\le 0 \\
\frac{1}{2}x+\frac{1}{2} & \text{ if } 0\le x\le 1 \text{.}
\end{cases}
\]
Since $f$ maps $[0,1]$ bijectively to $[1/2,1]$, we have $fF_{2,3}[0,1]f^{-1}=F_{2,3}[1/2,1]$.
\end{proof}

%-------------------------------------------------
\section{Characters of the Stein group}\label{sec:chars}

The first step in computing the BNSR-invariants of any group is to figure out its abelianization and space of characters. The abelianization was computed by Stein:

\begin{cit}\cite[Theorem~4.7]{stein92}\label{cit:rank2_ablnz}
The abelianization of $F_{2,3}$ is $\Z^4$, so $\Hom(F_{2,3},\R)\cong \R^4$ and $\Sigma(F_{2,3})\cong S^3$.
\end{cit}

To get an understanding of the structure of $\Hom(F_{2,3},\R)$, let us define some characters of $F_{2,3}$.

\begin{definition}[The characters $\lambda$ and $\rho$]
For $f\in F_{2,3}$ let $\lambda(f)$ be the natural logarithm of the (right) derivative of $f$ at $0$ and let $\rho(f)$ be the natural logarithm of the (left) derivative of $f$ at $1$. Since $0$ and $1$ are fixed by all elements of $F_{2,3}$, the chain rule says that $\lambda,\rho\in\Hom(F_{2,3},\R)$.
\end{definition}

\begin{definition}[The characters $\chi_0^2,\chi_0^3,\chi_1^2,\chi_1^3$]
Since $\{2,3\}$ is a basis for the free abelian group $\langle 2,3\rangle_{\Q^\times}$, we know that $e^{\lambda(f)}$ is of the form $2^p3^q$. Let $\chi_0^2(f)\defeq p$ and $\chi_0^3(f)\defeq q$ for these $p$ and $q$, so $\lambda=\ln(2)\chi_0^2+\ln(3)\chi_0^3$. Thus $\chi_0^2$ and $\chi_0^3$ are homomorphisms
\[
\chi_0^2,\chi_0^3 \colon F_{2,3} \to \Z \text{.}
\]
Similarly define
\[
\chi_1^2,\chi_1^3 \colon F_{2,3} \to \Z
\]
using $\rho$ instead of $\lambda$, so $\rho=\ln(2)\chi_1^2+\ln(3)\chi_1^3$.
\end{definition}

Now we want to show that $\chi_0^2,\chi_0^3,\chi_1^2,\chi_1^3$ form a basis of $\Hom(F_{2,3},\R)$, for which we just need to show they are linearly independent. The following will be useful for this, and in fact gives much more than we need at the moment, but will be useful later in this degree of generality.

\begin{lemma}\label{lem:special_elements}
For any $p,q\in\Z$ there exists $f\in F_{2,3}$ such that $\chi_0^2(f)=p$, $\chi_0^3(f)=q$, and $\rho(f)=0$. For any $0<r<1$ the element $f$ can be chosen to have support $\Supp(f)$ satisfying $(0,r)\subseteq \Supp(f)$. A similar result holds for the conditions $\chi_1^2(f)=p$, $\chi_1^3(f)=q$, and $\lambda(f)=0$.
\end{lemma}

\begin{proof}
First consider $p=1$ and $q=0$. For $f$ we can take any element of Thompson's group $F$ with initial slope $2$, final slope $1$, and support satisfying $(0,r)\subseteq \Supp(f)$ for whatever $0<r<1$ we want. Similarly if $p=0$ and $q=1$ then for $f$ we can take any element of the generalized Thompson group $F_3$ with initial slope $3$, final slope $1$, and support satisfying $(0,r)\subseteq \Supp(f)$. Now since $\chi_0^2$ and $\chi_0^3$ are homomorphisms we can compose elements from these special cases to build an $f$ for any $p$ and $q$. The last sentence of the claim follows by an analogous argument.
\end{proof}

\begin{corollary}\label{cor:LI}
The characters $\chi_0^2,\chi_0^3,\chi_1^2,\chi_1^3$ form a basis of $\Hom(F_{2,3},\R)$.
\end{corollary}

\begin{proof}
By Citation~\ref{cit:rank2_ablnz} we know that $\Hom(F_{2,3},\R)\cong \R^4$, so we just need to show that these four characters are linearly independent. It suffices to find, for each such character, an element of $F_{2,3}$ on which that character has a non-zero value and all other characters on the list have a zero value. This can easily be done following Lemma~\ref{lem:special_elements}.
\end{proof}

\subsection{The BNS-invariant}\label{sec:bns_stein}

Now that we have pinned down $\Hom(F_{2,3},\R)$ we can quickly compute $\Sigma^1(F_{2,3})$. The first BNSR-invariant $\Sigma^1$ is often called the \emph{BNS-invariant} (since Renz was only involved in developing $\Sigma^m$ for $m>1$ in \cite{bieri88}), and $\Sigma^1$ is often easier to compute than the higher $\Sigma^m$. In fact the BNS-invariant of $F_{2,3}$ is essentially already known, thanks to a result in the original Bieri--Neumann--Strebel paper \cite{bieri87}.

\begin{cit}\cite[Theorem~8.1]{bieri87}\label{cit:Sigma1}
$\Sigma^1(F_{2,3}) = \Sigma(F_{2,3})\setminus \{[\lambda],[\rho]\}$.
\end{cit}

It is clear that $F_{2,3}$ is irreducible and that $\lambda$ and $\rho$ are independent, in the sense of \cite{bieri87}, so the hypotheses of \cite[Theorem~8.1]{bieri87} are indeed met.

Since $\lambda=\ln(2)\chi_0^2+\ln(3)\chi_0^3$ and $\rho=\ln(2)\chi_1^2+\ln(3)\chi_1^3$, they are not discrete characters: their images in $\R$ are isomorphic to $\Z^2$. Thus we get:

\begin{corollary}\label{cor:rationally_full_Sigma1}
The kernel of any discrete character $\chi\colon F_{2,3} \to \Z$ is finitely generated. However, there exist characters $\chi\colon F_{2,3} \to \Z^2$ whose kernels are not finitely generated.
\end{corollary}

\begin{proof}
For the first claim, if $\ker(\chi)\le \ker(\lambda)$ then $F_{2,3}/\ker(\chi)\cong \Z$ surjects onto $F_{2,3}/\ker(\lambda)\cong \Z^2$, which is impossible. Similarly $\ker(\chi)$ cannot lie in $\ker(\rho)$. Since Citation~\ref{cit:Sigma1} says $[\lambda]$ and $[\rho]$ are the only elements of $\Sigma(F_{2,3})\setminus \Sigma^1(F_{2,3})$, Citation~\ref{cit:bnsr_fin_props} says $\ker(\chi)$ is finitely generated. For the second claim, just note that Citations~\ref{cit:bnsr_fin_props} and~\ref{cit:Sigma1} say that $\ker(\lambda)$ is not finitely generated.
\end{proof}

Citation~\ref{cit:Sigma1} says that every discrete character class lies in $\Sigma^1(F_{2,3})$. Later we will improve this to all discrete character classes lying in $\Sigma^\infty(F_{2,3})$ and all kernels of discrete characters being of type $\F_\infty$ (see Corollary~\ref{cor:rationally_full}).

%-------------------------------------------------
\section{The main computation}\label{sec:computation}

In this section we compute $\Sigma^m(F_{2,3})$ for all $m$. We will first set up some results involving ascending HNN-extensions.

\subsection{Ascending HNN-extensions}\label{sec:hnn}

For our purposes, a group $G$ is said to decompose as an \emph{ascending HNN-extension} if there exist a subgroup $B\le G$ called the \emph{base} and an element $t\in G$ called the \emph{stable element} such that $G=\langle B,t\rangle$, $t^{-1}Bt\le B$, and $B\cap\langle t\rangle=\{1\}$. (This is an ``internal'' ascending HNN-extension, and it is equivalent to the usual ``external'' definition thanks to \cite[Lemma~3.1]{geoghegan01}.) We write $G=B*_t$ to indicate these data. Call an ascending HNN-extension \emph{properly ascending} if $t^{-1}Bt\ne B$. In the future we may write $B^t$ for $t^{-1}Bt$.

We will frequently need to consider certain subgroups of $F_{2,3}$ that decompose as ascending HNN-extensions with base $F_{2,3}[1/2,1]$ (recall that $F_{2,3}[1/2,1]$ is the subgroup of $F_{2,3}$ consisting of elements with support contained in $[\frac{1}{2},1]$). Our next result describes this situation. Say a character of $F_{2,3}$ is \emph{left-based} if it is of the form $a\chi_0^2+b\chi_0^3$ for $a,b\in\R$. Call a character \emph{right-based} if it is of the form $a\chi_1^2+b\chi_1^3$ for $a,b\in\R$. Every character of $F_{2,3}$ is a sum of a unique left-based and unique right-based character. Call a character \emph{one-sided} if it is left- or right-based.

\begin{lemma}[Decompose discrete kernels]\label{lem:F23_hnn}
Let $0\ne\chi\in\Hom(F_{2,3},\R)$ be a discrete left-based character. Then there exists $t\in \ker(\chi)$ such that $\ker(\chi)$ decomposes as an ascending HNN-extension with base $F_{2,3}[1/2,1]$ and stable letter $t$.
\end{lemma}

\begin{proof}
Since $\chi$ is discrete, up to positive scaling we can assume $\chi=a\chi_0^2+b\chi_0^3$ for some coprime $a,b\in\Z$ (note that $0$ is coprime to $\pm1$, so even if one of $a$ or $b$ is $0$ this is still accurate). Since $\ln(2)/\ln(3)\not\in\Q$, we have either $b\ln(2)-a\ln(3)<0$ or $-b\ln(2)+a\ln(3)<0$. In the first case, choose $t\in F_{2,3}$ with $\chi_0^2(t)=b$ and $\chi_0^3(t)=-a$, such that $t$ has support satisfying $(0,3/4)\subseteq \Supp(t)$ (this is possible by Lemma~\ref{lem:special_elements}). In the second case, choose $t\in F_{2,3}$ with $\chi_0^2(t)=-b$ and $\chi_0^3(t)=a$, with $(0,3/4)\subseteq \Supp(t)$. In either case we have $\lambda(t)<0$ and $\chi(t)=0$. Now let $t'$ be any element of $\ker(\chi)$, so $a\chi_0^2(t')+b\chi_0^3(t')=0$. Since $a$ and $b$ are coprime, this implies that $\chi_0^2(t')=bk$ and $\chi_0^3(t')=-ak$ for some $k\in\Z$ (to be clear, this works even if one of $a$ or $b$ is $0$, since then the other one is $\pm1$). In particular $\chi_0^n(t^{\pm k})=\chi_0^n(t')$ for $n=2,3$, so $\lambda(t^{\pm k})=\lambda(t')$, which shows that $t'\ker(\lambda)=t^{\pm k}\ker(\lambda)$. By now we know that $\ker(\chi)$ is generated by $t$ and $\ker(\lambda)$. Since $(0,3/4)\subseteq\Supp(t)$, every element of $\ker(\lambda)$ is conjugate via a power of $t$ to an element of $F_{2,3}[1/2,1]$. Hence $\ker(\chi)$ is generated by $t$ and $F_{2,3}[1/2,1]$. It is also clear that $\langle t\rangle \cap F_{2,3}[1/2,1]=\{1\}$. Finally note that $F_{2,3}[1/2,1]^t\subseteq F_{2,3}[1/2,1]$ since $(0,3/4)\subseteq\Supp(t)$ and $\lambda(t)<0$. We conclude that $\ker(\chi)=F_{2,3}[1/2,1]*_t$.
\end{proof}

There is also a right-based version of the above, using $F_{2,3}[0,1/2]$, which works analogously to the left-based version.

Let us now record a number of general results about the interaction between ascending HNN-extensions and BNSR-invariants.

\begin{cit}\cite[Theorem~2.1]{bieri10}\label{cit:hnn_dead_base}
Let $G$ be a group of type $\F_\infty$ that decomposes as an ascending HNN-extension $G=B*_t$, with the base $B$ of type $\F_\infty$. Let $\chi\colon G\to\R$ be a character satisfying $\chi(B)=0$ and $\chi(t)>0$. Then $[\chi]\in\Sigma^\infty(G)$.
\end{cit}

\begin{cit}\cite[Theorem~2.3]{bieri10}\label{cit:hnn_live_base}
Let $G$ be a group of type $\F_\infty$ that decomposes as an ascending HNN-extension $G=B*_t$, with the base $B$ of type $\F_\infty$. Let $\chi\colon G\to\R$ be a character with $\chi(B)\ne 0$ and $[\chi|_B]\in\Sigma^\infty(B)$. Then $[\chi]\in\Sigma^\infty(G)$.
\end{cit}

Note that a special case of Citation~\ref{cit:hnn_live_base} is when the ascending HNN-extension is not properly ascending, i.e., $t^{-1}Bt=B$, which means $B$ is the kernel of a map from $G$ onto $\Z$. More generally, we have the following:

\begin{cit}\cite[Lemma~4.5]{meinert96}\label{cit:ker_to_Zn}
Let $G$ be a group of type $\F_\infty$ and $N$ a normal subgroup of type $\F_\infty$ with $G/N$ polycyclic-by-finite (for example free abelian). For $\chi\in\Hom(G,\R)$, if $[\chi|_N]\in\Sigma^\infty(N)$ then $[\chi]\in\Sigma^\infty(G)$.
\end{cit}

Finally, let us record the following easy fact:

\begin{observation}\label{obs:hnn_Finfty}
If $G=B*_t$ is an ascending HNN-extension and $B$ is of type $\F_\infty$, then so is $G$.
\end{observation}

\begin{proof}
Note that $G$ acts cocompactly on a contractible complex, namely the Bass--Serre tree for the ascending HNN-extension. Every vertex and edge stabilizer is isomorphic to $B$, hence is $\F_\infty$. The result now follows from standard facts about finiteness properties, see, e.g., \cite[Proposition~1.1]{brown87}.
\end{proof}

\subsection{The computation}\label{sec:subsec_computation}

Now we are ready to compute the $\Sigma^m(F_{2,3})$. First we focus on one-sided characters.

\begin{lemma}[Discrete one-sided]\label{lem:left_discrete}
Let $0\ne\chi\colon F_{2,3}\to \Z$ be a discrete one-sided character. Then $\ker(\chi)$ is of type $\F_\infty$, and so $[\chi]\in\Sigma^\infty(F_{2,3})$.
\end{lemma}

\begin{proof}
We will do the left-based version, and the right-based version is analogous. By Citation~\ref{cit:bnsr_fin_props} we just need to prove that $\ker(\chi)$ is of type $\F_\infty$. By Lemma~\ref{lem:F23_hnn}, $\ker(\chi)$ decomposes as an ascending HNN-extension with base $F_{2,3}[1/2,1]$. This base is isomorphic to $F_{2,3}$ by Lemma~\ref{lem:conjugator}, hence is of type $\F_\infty$. We conclude that $\ker(\chi)$ is of type $\F_\infty$ by Observation~\ref{obs:hnn_Finfty}.
\end{proof}

\begin{proposition}[One-sided]\label{prop:left_any}
Let $0\ne\chi\colon F_{2,3}\to \R$ be any one-sided character such that $[\chi]\ne[\lambda],[\rho]$. Then $[\chi]\in\Sigma^\infty(F_{2,3})$.
\end{proposition}

\begin{proof}
We will assume $\chi$ is left-based, and the right-based case works analogously.

First suppose $[\chi]\ne[-\lambda]$. Since $[\chi]\ne[\pm\lambda]$ we know $\chi$ and $\lambda$ are linearly independent, and hence span the $2$-dimensional subspace of $\Hom(F_{2,3},\R)$ consisting of all left-based characters. In particular we can choose $c>0$ such that $\chi+c\lambda$ is non-zero and discrete. By Lemma~\ref{lem:left_discrete}, $K\defeq \ker(\chi+c\lambda)$ is of type $\F_\infty$. By Citation~\ref{cit:ker_to_Zn} it now suffices to prove that the restriction of $\chi$ to $K$, which equals the restriction of $-c\lambda$ to $K$, lies in $\Sigma^\infty(K)$. This is equivalent to proving that $[-\lambda|_K]\in\Sigma^\infty(K)$. Since $\chi+c\lambda$ is non-zero, left-based, and discrete, by Lemma~\ref{lem:F23_hnn} we can choose $t\in K$ such that $K=F_{2,3}[1/2,1]*_t$. By the proof of Lemma~\ref{lem:F23_hnn} we can assume $\lambda(t)<0$. Now we have $-\lambda|_K(t)>0$ and $-\lambda|_K(F_{2,3}[1/2,1])=0$, so $[-\lambda|_K]\in\Sigma^\infty(K)$ by Citation~\ref{cit:hnn_dead_base}.

Now suppose $[\chi]=[-\lambda]$. Let $K\defeq \ker(\chi_0^3)$, so $K$ is of type $\F_\infty$ by Lemma~\ref{lem:left_discrete}. Hence it suffices to show that $[-\lambda|_K]\in\Sigma^\infty(K)$. Similar to the previous case, we can choose $t$ such that $K=F_{2,3}[1/2,1]*_t$ and $\lambda(t)<0$. Now $-\lambda|_K(t)>0$ and $-\lambda|_K(F_{2,3}[1/2,1])=0$, so $[-\lambda|_K]\in\Sigma^\infty(K)$ by Citation~\ref{cit:hnn_dead_base}.
\end{proof}

\begin{corollary}[Not a positive combo of $\lambda$ and $\rho$]\label{cor:all_in_Sigmainfty}
Let $0\ne\chi\colon F_{2,3}\to\R$ be a character satisfying $\chi=\chi_L+\chi_R$ for left-based $0\ne\chi_L$ and right-based $0\ne\chi_R$, such that either $[\chi_L]\ne[\lambda]$ or $[\chi_R]\ne[\rho]$. Then $[\chi]\in\Sigma^\infty(F_{2,3})$.
\end{corollary}

\begin{proof}
The two cases are analogous, so without loss of generality $[\chi_R]\ne[\rho]$. Let $H=\ker(\chi_0^2)$, so Lemma~\ref{lem:F23_hnn} says $H$ is an ascending HNN extension with base $F_{2,3}[1/2,1]$. Since $H$ is of type $\F_\infty$ by Lemma~\ref{lem:left_discrete}, Citation~\ref{cit:ker_to_Zn} says it suffices to show that the restriction of $\chi$ to $H$ is in $\Sigma^\infty(H)$. Since $H$ is an ascending HNN extension of $F_{2,3}[1/2,1]$, and $F_{2,3}[1/2,1]$ is of type $\F_\infty$ since it is isomorphic to $F_{2,3}$ (Lemma~\ref{lem:conjugator}), it suffices by Citation~\ref{cit:hnn_live_base} to show that the restriction of $\chi$ to $F_{2,3}[1/2,1]$ is in $\Sigma^\infty(F_{2,3}[1/2,1])$. This restriction coincides with the restriction of $\chi_R$ to $F_{2,3}[1/2,1]$. Identifying $F_{2,3}[1/2,1]$ isomorphically with $F_{2,3}$, the restriction of $\chi_R$ is identified with $\chi_R$ itself, so it remains to show that $[\chi_R]\in\Sigma^\infty(F_{2,3})$. Since $[\chi_R]\ne[\rho]$, Proposition~\ref{prop:left_any} says that indeed $[\chi_R]\in\Sigma^\infty(F_{2,3})$ and we are done.
\end{proof}

This quickly leads to a full computation of all the $\Sigma^m(F_{2,3})$:

\begin{theorem}\label{thrm:main}
Let $0\ne\chi\in\Hom(F_{2,3},\R)$. If $[\chi]=[\lambda],[\rho]$ then $[\chi]\not\in\Sigma^1(F_{2,3})$. If $\chi=a\lambda+b\rho$ for $a,b>0$ then $[\chi]\in\Sigma^1(F_{2,3})\setminus \Sigma^2(F_{2,3})$. Otherwise $[\chi]\in\Sigma^\infty(F_{2,3})$.
\end{theorem}

\begin{proof}
The computation of $\Sigma^1(F_{2,3})$ is in Citation~\ref{cit:Sigma1}. If $\chi=a\lambda+b\rho$ for $a,b>0$ then the fact that $[\chi]\not\in\Sigma^2(F_{2,3})$ follows from \cite[Theorem~A1]{kochloukova02} (which applies since $F_{2,3}$ has no non-abelian free subgroups \cite[Theorem~3.1]{brin85}). If $[\chi]=[-\lambda],[-\rho]$ then $[\chi]\in\Sigma^\infty(F_{2,3})$ by Proposition~\ref{prop:left_any}. In all other cases, $\chi$ satisfies the assumptions of Corollary~\ref{cor:all_in_Sigmainfty} and so $[\chi]\in\Sigma^\infty(F_{2,3})$.
\end{proof}

%-------------------------------------------------
\section{Applications and questions}\label{sec:apps}

Having computed all the $\Sigma^m(F_{2,3})$, let us discuss some applications. First we get a full classification of the finiteness properties of normal subgroups of $F_{2,3}$.

\begin{corollary}\label{cor:main}
Let $N$ be a non-trivial normal subgroup of $F_{2,3}$. Then $N$ is finitely generated if and only if $\lambda(N)\ne 0$ and $\rho(N)\ne 0$, $N$ is finitely presented if and only if $(a\lambda+b\rho)(N)\ne 0$ for all $a,b\ge0$, and $N$ is of type $\F_\infty$ if and only if it is finitely presented.
\end{corollary}

\begin{proof}
By Lemma~\ref{lem:normal_contain_comm}, $N$ contains $[F_{2,3},F_{2,3}]$. The result now follows by applying Citation~\ref{cit:bnsr_fin_props} to the computation in Theorem~\ref{thrm:main}.
\end{proof}

Since $F_{2,3}/[F_{2,3},F_{2,3}] \cong \Z^4$, we get a normal subgroup of $F_{2,3}$ for every subgroup of $\Z^4$ by taking the preimage in $F_{2,3}$. Hence the family of $N$ to which Corollary~\ref{cor:main} applies is quite robust.

Thanks to the non-discreteness of $\lambda$ and $\rho$, we get the following peculiar result, which improves Corollary~\ref{cor:rationally_full_Sigma1}.

\begin{corollary}\label{cor:rationally_full}
The kernel of any character $\chi\colon F_{2,3} \to \Z$ is of type $\F_\infty$. However, there exist characters $\chi\colon F_{2,3} \to \Z^2$ whose kernels are not even finitely generated.
\end{corollary}

\begin{proof}
The proof is similar to that of Corollary~\ref{cor:rationally_full_Sigma1}. For the second claim, just use $\chi=\lambda$. For the first claim, it suffices by Citation~\ref{cit:bnsr_fin_props} to show that for any character class $[\psi]$ not in $\Sigma^\infty(F_{2,3})$, the kernel of $\psi$ cannot contain $\ker(\chi)$ for $\chi$ discrete. Since $\Z$ cannot surject onto $\Z^2$, for this it suffices to see that the image of $\psi$ has rank at least $2$. By Theorem~\ref{thrm:main} we know $\psi=a\lambda+b\rho$ for some $a,b\ge 0$. Hence the image of $\psi$ either has rank $2$, if $a=0$ or $b=0$, or $4$, if $a,b>0$, so we are done.
\end{proof}

In particular, as discussed in the introduction, every discrete character class lies in $\Sigma^\infty(F_{2,3})$, but there exist (non-discrete) character classes that do not even lie in $\Sigma^1(F_{2,3})$. Additionally, Theorem~\ref{thrm:main} shows that $F_{2,3}$ provides an example of a type $\F_\infty$ group whose $\Sigma^\infty$ is determined by a polytope, but not an integral polytope. To the best of our knowledge this is the first group whose BNSR-invariants are known with any of these properties.

\subsection{Questions}\label{sec:questions}

Let us conclude by discussing some questions regarding generalizations of $F_{2,3}$. Given $S=\{n_1,\dots,n_s\}$ a subset of natural numbers with $n_i>1$ for all $i$, and any $r\in\N$, one can define the group $F_S^r$ of orientation-preserving piecewise linear homeomorphisms of $[0,r]$ with breakpoints in $\Z[\frac{1}{n_1\cdots n_s}]$ and slopes in $\langle S\rangle_{\Q^\times}$. This is the full generality in which Stein worked in \cite{stein92}, and the $F_S^r$ are sometimes called the \emph{Brown--Stein--Thompson groups}\footnote{According to \cite{bieri87}, Ken Brown considered the groups $F_S^1$ and could prove finite presentability. Brown did not publish this, presumably because it was superceded by his PhD student Melanie Stein's proof in \cite{stein92} that the $F_S^r$ are even of type $\F_\infty$. It seems appropriate then to call the $F_S^r$ the \emph{Brown--Stein--Thompson groups}.}. Even more generally, let $I\subseteq\R$ be an interval, $P$ a subgroup of the multiplicative group of positive real numbers, and $A$ a $\Z[P]$-submodule of the additive real numbers. Then we can define
\[
G(I;A,P) \defeq \{f\in \Homeo_+(I)\mid f \text{ is piecewise linear with slopes in } P \text{ and breaks in } A\}\text{.}
\]
For example,
\[
F_S^r = G\left([0,r];\langle S\rangle,\Z\left[\frac{1}{n_1\cdots n_s}\right]\right)\text{.}
\]
The fact that $A$ is invariant under multiplication by elements of $P$ ensures $G(I;A,P)$ is a group. These groups first appeared in \cite{bieri16} and are sometimes called the \emph{Bieri--Strebel groups} (the monograph \cite{bieri16} was not published until 2016, but was essentially completed in 1985).

Our arguments here for $\Sigma^m(F_{2,3})$ work equally well for groups of the form $F_{2,n}^r$, but as soon as $|S|>2$ and/or $2\not\in S$, things become significantly more difficult. For example it is not even clear what the abelianization of $F_{3,4,5}^1$ is. The abelianizations for $|S|\le 2$ are given in \cite{stein92}, and if $S$ satisfies that $n_1-1$ divides $n_i-1$ for all $i$ then Wladis's presentation of $F_S^1$ in \cite{wladis12} could in theory be abelianized, but already for $F_{3,4,5}^1$ none of the existing results apply. There is a topological procedure outlined in \cite{stein92} to abelianize the $F_S^r$, but to actually do it is a huge technical challenge. Given that abelianizing $F_S^r$ is already hard, it is clear that computing $\Sigma^m(F_S^r)$ in full generality is a very difficult task. As for the Bieri--Strebel groups $G(I;A,P)$, it is not even clear how often they are finitely presented, much less of type $\F_\infty$, so we are nowhere close to understanding their BNSR-invariants.

We do conjecture that the $\Sigma^m(F_S^r)$ behave similarly to $\Sigma^m(F_{2,3})$ and $\Sigma^m(F)$ in the following sense:

\begin{conjecture}
For any $S$ and $r$ we have $\Sigma^1(F_S^r)=\Sigma(F_S^r)\setminus\{[\lambda],[\rho]\}$, $\Sigma^2(F_S^r)=\Sigma(F_S^r)\setminus\{[a\lambda+b\rho]\mid a,b\ge 0\}$, and $\Sigma^\infty(F_S^r)=\Sigma^2(F_S^r)$.
\end{conjecture}

Let us make one final remark: it is notable how often the behavior $\Sigma^2=\Sigma^\infty$ holds for ``globally defined'' groups related to Thompson's group $F$, such as the $F_n$ \cite{kochloukova12,zaremsky17F_n}, braided $F$ \cite{zaremsky18}, the Lodha--Moore group \cite{lodha20BNSR}, and now the Stein group $F_{2,3}$ (and conjecturally all the Brown--Stein--Thompson groups $F_S^r$). It would be nice to find some deeper understanding of this phenomenon. In particular it would be interesting to know whether as soon as a Bieri--Strebel group $G(I;A,P)$ is of type $\F_\infty$ it satisfies $\Sigma^2(G(I;A,P))=\Sigma^\infty(G(I;A,P))$. Note that there do exist groups of piecewise linear homeomorphisms that have different $\Sigma^2$ and $\Sigma^\infty$ (or even different $\Sigma^m$ and $\Sigma^\infty$ for arbitrarily large $m$), for instance direct products of copies of $F$ \cite[Theorem~3.2]{bieri10}, but these are somewhat ad hoc, and do not have a nice ``global'' definition like the Bieri--Stebel groups. In an opposite direction, to the best of our knowledge, no example is known of a type $\F_\infty$ group $G$ for which the inclusion $\Sigma^m(G)\subseteq \Sigma^{m-1}(G)$ is proper for every $m\in\N$, so it would be interesting to hunt for such an example in the world of groups of piecewise linear homeomorphisms of $[0,1]$.

\bibliographystyle{alpha}

\end{document}